\documentclass{amsart}

\usepackage{
amsfonts,
latexsym,
amssymb,
}

\newcommand{\labbel}{\label}

\newtheorem*{theorem*}{Theorem}
\newtheorem*{proposition*}{Proposition}
\newtheorem*{corollary*}{Corollary}

\theoremstyle{definition}

\theoremstyle{remark}

\newcommand{\m}{\mathfrak}
\begin{document}

\title{Initial $\lambda$-compactness in linearly ordered spaces}

\author{Paolo Lipparini} 
\address{Dipartimento di Matematica\\ Viale degli Spazi Topologici\\II Universit\`a di Roma (Tor Vergata)\\I-00133 ROME ITALY}
\urladdr{http://www.mat.uniroma2.it/\textasciitilde lipparin}
\email{lipparin@axp.mat.uniroma2.it}



\begin{abstract}
We show that a linearly ordered topological space
 is initially $\lambda$-compact if and only if it
 is $\lambda$-bounded, that is, every set of cardinality
$\leq \lambda$ has compact closure. As a consequence,
every product of  initially $\lambda$-compact
linearly ordered topological spaces is initially $\lambda$-compact.
\end{abstract} 
 
\maketitle

A topological space is \emph{initially $\lambda$-compact} if
every open cover by at most $\lambda$ sets has a finite subcover. 
According to a celebrated theorem,  Stephenson  and  Vaughan \cite[Theorem 1.1]{SV}, if $\lambda$ is a strong limit singular cardinal, then 
every product of initially $\lambda$-compact topological spaces 
is still  initially $\lambda$-compact. We prove a much stronger result
for products whose factors are linearly ordered topological spaces: for such spaces, the above theorem holds for \emph{every} infinite cardinal $\lambda$.
In fact, our proof works for \emph{generalized  ordered spaces},
for short, \emph{GO spaces}, that is,
  Hausdorff  spaces  equipped 
with  a  linear  order  and  with  a  base  of  order-convex  sets.
See, e.~g.,  Bennet and Lutzer \cite{BL} for more information about 
GO spaces.

%In fact, 
We shall prove a chain
of equivalences which involve
several notions, such as $\lambda$-boundedness, 
$D$-compactness, $D$-pseudocompactness,
conditions asking for the existence of ``complete 
accumulation points'' of sequences of open sets,
and a condition simply asking that strictly increasing or 
decreasing sequences indexed by a regular cardinal converge.
To state our theorem in such a full generality we need to recall
some definitions.
 If $D$ is an ultrafilter over some set $I$, then a topological space $X$ is said to be \emph{$D$-compact}
  if every $I$-indexed sequence $(x_i) _{i \in I} $ of elements of $X$ 
 \emph{$D$-converges} to some $x \in X$, that is,
 $\{ i \in I \mid x_i \in U\} \in D$,
for every open neighborhood $U$ of $x$.
The space $X$ is said to be \emph{$D$-pseudocompact}
  if every $I$-indexed sequence $(O_i) _{i \in I} $ of nonempty open subsets of $X$ 
has some  \emph{$D$-limit point} in $X$, that is, there is
 some $x \in X$ such that 
 $\{ i \in I \mid  U \cap O_i \not= \emptyset \} \in D$,
for every open neighborhood $U$ of $x$.
If $\beta$ is a limit ordinal, we say that a sequence $(x_ \gamma ) _{ \gamma < \beta } $ of elements of a topological space 
\emph{converges} to some point $x$   if,
for every neighborhood $U$ of $x$,
there is $\gamma < \beta $ such that 
$x _{ \gamma '} \in U $, for every $\gamma' > \gamma $.

\begin{theorem*} \labbel{thm}
For every infinite cardinal $\lambda$, and every GO
space $X$, the following conditions are equivalent.
 \begin{enumerate}
   \item 
 $X$ is initially $\lambda$-compact.
   \item 
 $X$ is \emph{weakly initially $\lambda$-compact}, that is, every open cover 
of $X$ by at most $ \lambda$ sets has a finite subcollection with dense union. 
\item
For every infinite (equivalently, every infinite regular) cardinal $ \nu \leq \lambda $,
and every family  $(O_ \gamma ) _{\gamma < \nu}$ of $\nu$ open 
 nonempty sets of $X$,   
 there is $x \in X$ such that 
$|\{\gamma < \nu \mid O_ \gamma  \cap U \not= \emptyset \}|= \nu$, 
for every neighborhood $U$ of $x$.   

In the  above condition we can equivalently ask either that the $O_ \gamma $'s
are pairwise disjoint, or  that $O_ \gamma \subseteq O _{ \gamma '} $,
for $\gamma> \gamma '$.  
\item
For every infinite regular cardinal $ \nu \leq \lambda $,
and every strictly increasing (resp., strictly decreasing) $\nu$-indexed sequence
of elements of $X$, the sequence has a supremum (resp., an infimum) to which it
converges.
\item 
$X$ is $D$-compact, for every ultrafilter $D$ over any set of cardinality 
$ \leq\lambda$.
\item 
$X$ is $D$-pseudocompact, for every ultrafilter $D$ over any set of cardinality 
$ \leq\lambda$.
\item
$X$ is \emph{$\lambda$-bounded}, that is, every subset of cardinality
$\leq \lambda$ has compact closure.
  \end{enumerate} 
 \end{theorem*}

\begin{proof}
We shall first prove the following chain of equivalences:
(1) $\Rightarrow $  (2) $\Rightarrow $  (3)$_{\rm reg}$ $\Rightarrow $ 
(4) $\Rightarrow $  (5) $\Rightarrow $  (1), where by
  (3)$_{\rm reg}$ we denote the condition (3) restricted to 
regular $\nu$'s.
By the way, notice that the implication (1) $\Rightarrow $  (4) is trivial, hence the reader interested
only in the proof of the equivalence of (1), (4), (5) and (7) could skip the 
next three passages.
 
(1) $\Rightarrow $  (2) is trivial. 

(2) $\Rightarrow $  (3) is known, and true
for every topological space.
First, we prove here (2) $\Rightarrow $  (3)$_{\rm reg}$.
 Suppose that (2) holds, and that the conclusion of   (3)$_{\rm reg}$
fails. Then, since $\nu$ is regular,
 for every $x \in X$ we can choose an open  neighborhood  $U_x$ of $x$ and some
$ \delta _x < \nu$ such that   
$O_ \gamma  \cap U_x = \emptyset$, for every $ \gamma > \delta_x$.    
For every $ \delta < \nu$, let $V_ \delta =  \bigcup _{\delta_x = \delta}  U_x$.
Thus $(V_ \delta ) _{ \delta < \nu}$ is  an open cover of $X$ by $\leq \lambda$ sets,
hence, by (2), it has a finite subcollection with dense union, say,
$ V _{ \delta_1}$, \dots,  $ V _{ \delta_n}$. 
If $ \gamma = \sup \{  \delta_1, \dots, \delta_n \} + 1$, 
then $O_ \gamma  \cap (V _{ \delta_1} \cup \cdots \cup  V _{ \delta_n})=
\emptyset$, a contradiction, since $O_ \gamma$ is nonempty. 

 (3)$_{\rm reg}$ $\Rightarrow $  (4) is easy. Suppose that
$\nu$ is a regular cardinal, and
that $(x_ \gamma) _{ \gamma < \nu}$ is, say, a strictly increasing sequence.
For $\gamma < \nu $, define $O_ \gamma = (x_ \gamma, x _{\gamma +2} ) 
= \{ x \in X \mid x_ \gamma < x < x _{ \gamma +2}  \} $.
The $O_ \gamma$'s are open  and nonempty, since $x _{\gamma +1} \in O_ \gamma$.
It is immediate to see that the $x$ given by  (3)$_{\rm reg}$ is a supremum of
   $(x_ \gamma) _{ \gamma < \nu}$ to which the sequence converges. 

If the open sets in  (3)$_{\rm reg}$ are required to be disjoint,
simply take only the ``even'' above sets, namely,
for $\gamma= \alpha + n$, with $\alpha=0$ or $\alpha$  limit,
let   $O_ \gamma = (x _{\alpha + 2n} , x _{ \alpha  +2n+2} ) $.

If the sequence of open sets in  (3)$_{\rm reg}$ is required to be $ \subseteq $-decreasing, take
 $O _ \gamma = \bigcup _{ \gamma '> \gamma } (x_ \gamma, x _{\gamma'} )$.
Thus the proof of  (3)$_{\rm reg}$ $\Rightarrow $  (4) is complete in each case.

Next, we  concentrate on the proof of (4) $\Rightarrow $  (5).
Suppose that (4) holds, and that
 $D$ is an ultrafilter  over some set $I$ of cardinality $\leq \lambda $.
Let $(x_i) _{i \in I} $  be an $I$-indexed sequence of elements of $X$: we have to show that
$(x_i) _{i \in I} $   $D$-converges in $X$. 
Without loss of generality, we can suppose that,
for every $x \in X$, 
$\{ i \in I \mid x_i= x  \} \not\in D$,
since, otherwise,
clearly $(x_i) _{i \in I} $   $D$-converges to $x$, and we are done.
Since $D$ is an ultrafilter, and $X$ is linearly ordered,
then, for each $j \in I$, either 
$A_j = \{ i \in I \mid x_j<x_i\} \in D$,
or  
$B_j=\{ i \in I \mid x_i < x_j\} \in D$.
Let 
$L = \{ x_j \mid A_j \in D\}$ 
and
$R = \{ x_j \mid B_j \in D\}$;
thus, in particular,
$L \cup R = \{ x_i \mid i \in I \}$,
hence, again since $D$ is an ultrafilter,
then either 
$\{ i \in I \mid x_i \in L \} \in D$,
or 
$\{ i \in I \mid x_i \in R \} \in D$.
Suppose that $\{ i \in I \mid x_i \in L \} \in D$;
the other case is treated in a symmetrical way.
Notice that $L$ cannot have a maximum:
if  $x_j \in L$,
then, by definition, 
 $A_j = \{ i \in I \mid x_j<x_i\} \in D$,
and if $x_j$ is a maximum for $L$,
this contradicts $\{ i \in I \mid x_i \in L \} \in D$.
Hence $L$
has infinite cofinality  $ \leq \lambda $,
since $|L| \leq |I| \leq \lambda $.
Say, $\nu$ is the cofinality of $L$;
thus, we can choose a strictly increasing sequence 
$(x _{i_ \alpha } ) _{ \alpha < \nu} $ 
cofinal in $L$. By (4), this sequence has 
a supremum to which it converges,
call it $\ell$;
in particular, $\ell$ 
is also the supremum of $L$, and
every neighborhood
of $\ell$ contains 
a convex interval of the form
$( x _{i_ \alpha },\ell]$,
for some $\alpha < \nu$.
But this soon implies that  
$(x_i) _{i \in I} $   $D$-converges to $\ell$;
indeed, for every $\alpha < \nu$,
$\{ i \in I \mid x_i \in ( x _{i_ \alpha },\ell] \} 
\subseteq A _{i_ \alpha } \cap \{ i \in I \mid x_i \in L \} \in D$,
since $D$ is a filter.
The proof of the implication (4) $\Rightarrow $  (5) is thus complete.

(5) $\Rightarrow $  (1) is nowadays a well-known standard argument, and, in fact, 
the implication holds for every 
topological space. See, e.~g.,  \cite[implication (7) in Diagram 3.6]{St}. 

We have proved the equivalence of (1), (2),  (3)$_{\rm reg}$, (4) and (5). Now
notice that (5) $\Rightarrow $ (6) $\Rightarrow $  (3) $\Rightarrow $  
 (3)$_{\rm reg}$
 are trivial: to show (6) $\Rightarrow $  (3), just consider, for every 
$\nu \leq \lambda $, some uniform ultrafilter 
over $\nu$.

Finally, the equivalence of (5) and (7) is well-known, and holds for every 
Hausdorff regular space, \cite[Theorems 5.3 and 5.4]{Sa}
(recall that it can be proved that every GO space is regular). 
\end{proof}

\begin{corollary*} \labbel{cor}
Suppose that $X$ is a product of topological spaces, and that all
 factors but at most one are GO spaces.
Then $X$ is  initially $\lambda$-compact
if and only if each factor is initially $\lambda$-compact.
 \end{corollary*} 

\begin{proof}
One implication is trivial.
For the other direction, by the equivalence of (1) and (7) in the Theorem,
all but at most one factor are $\lambda$-bounded.
It is well-known that a product of regular $\lambda$-bounded spaces
is still $\lambda$-bounded \cite[Theorem 5.7 and implications (1), (1$'$) in Diagram 3.6]{St},
and that a product of a $\lambda$-bounded space with an 
  initially $\lambda$-compact space is  initially $\lambda$-compact
\cite[Theorem 5.2 and implications (1), (2) in Diagram 3.6]{Sa}. Hence the corollary follows by first grouping together
the GO spaces, and then, in case,  multiplying their product
with the possibly non GO factor.
 \end{proof}  

The particular cases of the above
Theorem  and Corollary
when $\lambda= \omega $ appeared in 
 Sanchis and  Tamariz-Mascar\'ua \cite [Section 2]{ST}, 
or are immediate consequences of the statements 
there.

By slightly more elaborate arguments, the proof of
the implication (4) $\Rightarrow $  (5) in the above theorem gives the following
proposition. Recall that an ultrafilter $D$ over $I$ is
$\nu$-decomposable if there is some function 
$f: I \to \nu$ such that $f ^{-1}(A) \not\in D $,
for every $A \subset \nu$ with $|A| < \nu$.

\begin{proposition*} \labbel{prop} 
Suppose that $\lambda$ is an infinite cardinal, $X$ is a GO space,
and $D$ is an ultrafilter over some set of cardinality $ \leq\lambda$.
Suppose further that, for every regular cardinal $\nu \leq \lambda $, at least one of the following 
conditions are satisfied: 
 \begin{enumerate}
 \item
For every strictly increasing (or strictly decreasing) $\nu$-indexed sequence
of elements of $X$, the sequence has a supremum (or an infimum) to which it
converges, or
\item 
$D$ is not $\nu$-decomposable.
  \end{enumerate} 
Then $X$ is $D$-compact.

In particular, $X$ is $D$-compact if 
$D$ is $\mu $-complete and $X$ is $[\mu, \lambda ]$-compact. 
\end{proposition*} 

Recall that an ultrafilter $D$ is \emph{$\mu $-complete} if every intersection
of $<\mu $ members of $D$ is still in $D$.
A topological space is \emph{$[\mu, \lambda ]$-compact}
if every open cover by at most $\lambda$ sets has a subcover
by $<\mu $ sets.
Details shall be presented elsewhere.

\bibliographystyle{plain}

\begin{thebibliography}{STM}


\bibitem[BL]{BL}  H. R. Bennett and D. J. Lutzer,
   \emph{Recent developments in the topology of ordered spaces}, in
 \emph{Recent progress
   in general topology, II} (M. Hu\v sek and J. van Mill. eds.), 83--114, North-Holland, Amsterdam, 2002.





\bibitem[Sa]{Sa} V. Saks, \emph{Ultrafilter invariants in   topological spaces}, Trans. Amer. Math. Soc. \textbf{241} (1978), 79--97.

 



\bibitem[STM]{ST} M. Sanchis and  A. Tamariz-Mascar\'ua,
   \emph{A note on $p$-bounded and quasi-$p$-bounded subsets},
   Houston J. Math. \textbf{28} (2002), 511--527.

 


\bibitem[St]{St} R. M.   Stephenson, Jr,
   \emph{Initially $\kappa$-compact and related spaces},
in \emph{Handbook of
   set-theoretic topology} (K. Kunen and J. E. Vaughan, eds.), 603--632, North-Holland, Amsterdam, 1984.


\bibitem[SV]{SV} R. M. Stephenson, Jr. and J. E.  Vaughan, 
   \emph{Products of initially $\m m$-compact spaces},
   Trans. Amer. Math. Soc. \textbf{196} (1974), 177--189.

 

 

\end{thebibliography}

\end{document}